\documentclass[a4paper,11pt]{amsart}
\usepackage[all]{xy}
\usepackage{amscd}
\usepackage{amsfonts}
\usepackage{a4wide}
\usepackage{amsmath, amsthm}
\usepackage{amssymb}
\usepackage{mathtools}
\usepackage{enumerate}
\usepackage{fullpage}
\usepackage[applemac]{inputenc}
\usepackage[pdftex,breaklinks,colorlinks,citecolor=blue]{hyperref}

\newtheorem{theorem}{Theorem}[section]

\newtheorem{proposition}[theorem]{Proposition}
\newtheorem{lemma}[theorem]{Lemma}

\providecommand{\keywords}[1]

\title{Iterated constructions of completely normal polynomials}
\author{Anibal Aravena $^\dag$}
\address{ Anibal Aravena. Pontificia Universidad Cat\'olica de Chile,  Facultad de Matem\'aticas, Vicu\~na Mackenna 4860, Santiago, Chile.
}
\email{akaravena@uc.cl}
\thanks{\dag Pontificia Universidad Cat\'olica de Chile}
\thanks{The author was partially supported by FONDECYT 1171329
}

\begin{document}
 \maketitle
\begin{abstract}
    The $R_{\sigma,t}$-transform introduced by Bassa and Menares can be used to construct families of irreducible polynomials in $\mathbb{F}_q[x]$. This iterative construction is a generalization of Cohen's $R$-transform. For this transform,  Chapman proved that under some conditions, the polynomials in the resulting family   are completely normal. In this paper we establish conditions ensuring that the polynomials obtained by using the $R_{\sigma,t}$-transform are completely normal polynomials and we give a simple proof of Chapman's result.
 \end{abstract}\hspace{10pt}
 
\noindent \textbf{Keywords:} Finite Fields; normal polynomials; $R$-transform. 
\section{Introduction}
Let $q$ be a  prime power number and $n$ a positive integer. An element $\alpha \in \mathbb{F}_{q^n}$ is said normal in the extension $\mathbb{F}_{q^n}/\mathbb{F}_q$ if the elements $\{\alpha, \alpha^q,\ldots , \alpha^{q^{n-1}}\}$ are linearly independent over $\mathbb{F}_q$. Furthermore, if for each divisor $d$  of $n$, $\alpha$ is normal in the extension $\mathbb{F}_{q^n}/\mathbb{F}_{q^d}$, then $\alpha$ is called a completely normal element in the extension $\mathbb{F}_{q^n}/\mathbb{F}_{q}$.  Let $f$ be an irreducible polynomial in $\mathbb{F}_q[x]$, we say that $f$ is normal (completely normal) over $\mathbb{F}_q[x]$ if any of its roots is normal (completely normal) in the extension $\left.\mathbb{F}_{q^{\deg f}}\right/\mathbb{F}_q$.

Starting with a polynomial $g_0(x)\in \mathbb{F}_q[x]$, we focus on an iterative construction introduced in \cite{bassa2019rtransform} to get a family $\{g_k\}_{k\geq 0}$ of polynomials in $\mathbb{F}_q[x]$ of growing degree. Our aim is to find suitable hypothesis ensuring that $\{g_k\}_{k\geq 0}$ are all completely normal polynomials over $\mathbb{F}_q[x]$. To define it and state the basic results we need some previous notations. Let 
\begin{equation*}
    \text{GL}_2(\mathbb{F}_q)=\left\{\sigma =\begin{pmatrix}
a & b\\ 
c & d
\end{pmatrix} \in M_2(\mathbb{F}_q): ad-bc\neq 0\right\}.
\end{equation*}
For $   \alpha \in \overline{\mathbb{F}}_q\cup \{\infty\}, \sigma=\begin{pmatrix}
a & b\\ 
c & d
\end{pmatrix} \in {\rm GL}_2(\mathbb F_q)$, we  define the operation  $\sigma \cdot \alpha$   by
\begin{equation*}
 \sigma\cdot\alpha =\begin{cases}
\frac{a\alpha+b}{c\alpha+d} & \text{ if }  c\alpha + d \neq 0 \\ 
\infty & \text{ if } c\alpha + d =0
\end{cases}, \quad \sigma \cdot \infty=\begin{cases}
\frac{a}{c} & \text{ if } c\neq 0 \\ 
\infty & \text{ if } c=0 
\end{cases}.
\end{equation*}
Also, for a degree $n$ polynomial $f$ in $\mathbb{F}_q[x]$, we define the polynomial $P_\sigma(f)$ by
\begin{equation*}
P_{\sigma}(f)(x)=(cx+d)^nf\left(\frac{ax+b}{cx+d}\right).
\end{equation*}

Also, for any positive integer $t$, let $S_t:\mathbb{F}_q[x]\to \mathbb{F}_q[x]$ be the map given by $S_t(f)(x)=f(x^t)$.\\
Finally, for  a dregee $n$ polynomial $g$, we define the ${R_{\sigma,t}}$-transform by 
\begin{equation*}
    g^{R_{\sigma,t}}(x)=P_{\sigma^{-1}}\circ S_t\circ P_{\sigma}(g)(x).
\end{equation*}
If $g(a/c)\neq 0$, we define the element $\eta(g;\sigma) \in \mathbb{F}_q$ by 
\begin{equation*}
\eta(g;\sigma)=(\sigma^{-1}\cdot \infty)^n\dfrac{g(\sigma \cdot 0)}{g(\sigma\cdot \infty)}=\left(-\dfrac{d}{c}\right)^n\cdot g\left( \dfrac{b}{d}\right)\cdot g\left(\dfrac{a}{c}\right)^{-1}
\end{equation*}
with  the convention $\eta(g;\sigma)=(-d/a)^n\cdot g(b/d)$ if $c=0$ and $ \eta(g;\sigma)=(-b/c)^n\cdot g(a/c)^{-1}$ if $d=0$. 

The following Theorem is about the irreducibility of the iterated sequence $\{g_k\}_{k\geq0}$ constructed using the $R_{\sigma,t}$-transform.
\begin{theorem}[\cite{bassa2019rtransform},Theorem 1.2] \label{Rsigma}
Consider a finite field $\mathbb F_q$. Let $t\geq 2$ be an integer such that every prime factor of $t$ divides $q-1$. Let  $\sigma=\begin{pmatrix}
a & b\\ 
c & d
\end{pmatrix} \in {\rm GL}_2(\mathbb F_q)$.  Let $g(x)\neq x-a/c$ be a monic irreducible polynomial in $\mathbb F_q[x]$ of degree $n$. If $q\equiv 3 \mod 4$ and $t$ is even assume moreover that $n$ is even.

Assume that,  for all prime numbers $\ell|t$, the element $\eta(g;\sigma)$ is not an $\ell$-th power in $\mathbb F_q$. Define $g_0=g$ and let $g_k=g_{k-1}^{R_{\sigma,t}}$ for $k\geq 1$. Then $\{g_k\}_{k\geq 0}$ forms an infinite sequence of irreducible polynomials, with $\deg g_k=t^k\cdot n$.
\end{theorem}
An important particular case is the $R$-transform, previously introduced by S.Cohen \cite{Cohen1992TheEC}:
\begin{equation*}
    g^R(x)=(2x)^{\deg g}\cdot g\left(\frac{1}{2}\left(x+\frac{1}{x}\right)\right).
\end{equation*}
This corresponds to use  $t=2$, $\sigma^*=\begin{pmatrix}
1 & 1\\ 
1 & -1
\end{pmatrix}$ and  we have that $\eta(\sigma
^*,g)=g(1)g(-1)^{-1}$. Moreover, we have the following result.
\begin{theorem}[\cite{CHAPMAN19971},Theorem 1]\label{chaptheo}
Let $q\equiv 1 \mod 4$ be a prime power, and let $g_1=x^2+ax+1$ be an irreducible quadratic polynomial. Define $g_k$ recursively by $g_{k+1}=g_{k}^R$. If $\alpha_k$ is a zero of $g_k$, then $\alpha_k$ is a  completely normal element in the extension $\mathbb{F}_{q^n}/\mathbb{F}_q$.
\end{theorem}
Our main Theorem is a generalization of this result for the $R_{\sigma,t}$-transform 
\begin{theorem} \label{theonormal}
Let $g(x)=x-A\in \mathbb{F}_q[x]$  with $A\neq -a/c $, $t\geq 2,\sigma\in  {\rm GL}_2(\mathbb F_q)$ be  as in Theorem \ref{Rsigma}. Let $\{g_k\}_{k\geq 0}$ the sequence constructed by iteration of  the $R_{\sigma,t}$-transform. Suppose  $dc\neq 0$, then we have the following results:
\begin{enumerate}[i)]
    \item If $ab=0$, then  $\{g_k\}_{k\geq 0}$ is a sequence of completely normal polynomials over $\mathbb{F}_q[x]$.
    \item If $ab\neq 0$, and the element $\frac{bc}{ad}$ is an $\ell$-th  power for some prime $\ell$ dividing $t$, then  $\{g_k\}_{k\geq 0}$ is a sequence of completely normal polynomials over $\mathbb{F}_q[x]$.
    \item If $ab\neq 0$, $d=-c$ and $A\neq 0$, then $\{g_k\}_{k\geq 0}$ is a sequence of completely normal polynomials over $\mathbb{F}_q[x]$
\end{enumerate}
\end{theorem}
In section 4, we show how to recover Chapman's result from Theorem \ref{theonormal}. In section 5, we prove a Theorem more general that Theorem \ref{theonormal} in that the starting polynomial does not need to be linear. We also provide explicit examples.
\section{Preliminaries in Algebra}
This section is a brief summary about the basic results that we will use in the proof of Theorem \ref{theonormal}. For more details and proofs see \cite{BLAKE1997227} and \cite{Semaev_1989}.

Let $\tau: \overline{\mathbb{F}}_q \to \overline{\mathbb{F}}_q$ be the Frobenius map defined by $\tau (\alpha)=\alpha^q$. For every $n \in \mathbb{N}$, $\tau $ induces an automorphism of $\mathbb{F}_{q^n}$ that fixes $\mathbb{F}_q$. 
Let $f(x)=\sum_{i=0}^ma_ix^i$ be a polynomial in $\mathbb{F}_q[x]$, we define $f(\tau)$ $\in \text{End}(\mathbb{F}_{q^n})$  by $f(\tau)(\alpha)=\sum_{i=0}^ma_i\tau^i(\alpha)$.

Assume that $(n,q)=1$ and let
\begin{equation*}
    x^n-1=p_1p_2\ldots p_r
\end{equation*}
be the factorization of $x^n-1$ into irreducible factors over $\mathbb{F}_q[x]$. Samaev \cite{Semaev_1989} gives the following characterisation of normal elements in the extension $\mathbb{F}_{q^n}/\mathbb{F}_q$. 
\begin{theorem}[\cite{Semaev_1989}, Lemma 2.1]
Let $W_l=\{\alpha\in \mathbb{F}_{q^n} : p_l(\tau)(\alpha)=0\}$ for $l=1,\ldots,r$. Then 
\begin{equation*}
    \mathbb{F}_{q^n}=\bigoplus_{l=1}^rW_l
\end{equation*}
is a direct sum, where each $W_l$ is a vector space  over $\mathbb{F}_q$ of dimension $\deg p_l$. Furthermore, an element $\alpha\in \mathbb{F}_{q^n}$  is normal if and only if for $\alpha=\sum_{l=1}^r \alpha_l$ with $\alpha_l \in W_l$, we have that  $\alpha_l\neq 0$ for all $l=1,\ldots,r$.
\end{theorem}
Each $V_l$ is irreducible $\tau$-invariant i.e. $\tau (W_l) \subset W_l$ and it has no proper $\tau$-invariant subspace. Also the decomposition is unique, namely, if there exist $\{V_l\}_{l=1}^s$ irreducible $\tau$-invariant sub-spaces such that
\begin{equation*}
    \mathbb{F}_{q^n}=\bigoplus_{l=1}^sV_l,
\end{equation*}
then $s=r$ and, after rearranging the order of $V_l$'s if necessary, $V_l=W_l$ for $l =1,\ldots,r$.

We define an equivalence relation $\sim$ on $\{0,1\ldots,n-1\}$  by $l_1\sim l_2$ iff $l_1\equiv l_2q^{i} \mod n$ for some $i$. For $l \in \{0,1,\ldots,n-1\}$, let $M_l$ be its equivalence class. Then 
\begin{equation*}
    |M_l|=\min \{m \geq 1 : l\equiv lq^{m} \mod n\}
\end{equation*}
and there exist $S\subset \{0,1\ldots,n-1\}$ such that $\cup_{l\in S} M_l = \{0,1\ldots,n-1\}$ and $M_{l}\cap M_j=\emptyset $ if $l,j\in S$ and $l\neq j$. 
\begin{theorem}[\cite{Semaev_1989}, Theorem 2.2]  \label{theorepre}
Let $x^n-A\in \mathbb{F}_q[x]$ be an irreducible polynomial. Let $\theta\in \mathbb{F}_{q^n}$ be a root of $x^n-A$ and let $S,M_l$ be the sets defined previously. For $l \in S$, let $V_l$ be the sub-space of $\mathbb{F}_{q^n}$ spanned over $\mathbb{F}_q$ by the elements $\{\theta^m:m \in M_l\} $. Then each $V_l$ is an irreducible $\tau $-invariant subspace and
\begin{equation*}
    \mathbb{F}_{q^n}=\bigoplus_{l\in S}V_l.
\end{equation*}
In particular, an element $\alpha\in \mathbb{F}_{q^n}$ is normal if and only if for $\alpha=\sum_{l\in S} \alpha_l$ with $\alpha_l \in V_l$, we have that $\alpha_l\neq 0$ for all $l\in S$.
\end{theorem}

\section{Main Lemma}
If $f$ is a degree $n$ polynomial in $\mathbb{F}_q[x]$, the polynomial $P_\sigma(f)$ is given by
\begin{equation*}
P_{\sigma}(f)(x)=(cx+d)^nf\left(\frac{ax+b}{cx+d}\right).
\end{equation*}
An  useful result about this transformation is the following Proposition.
\begin{proposition} [\cite{bassa2019rtransform}, Proposition 3.1 and Lemma 3.2] \label{theopsigma}
Let $f\in \mathbb{F}_q[x]$ be a polynomial of degree $n$ and $\sigma=\begin{pmatrix}
a & b\\ 
c & d
\end{pmatrix} \in {\rm GL}_2(\mathbb F_q)$.  Then 
\begin{enumerate}[i)]
    \item $P_\sigma(f)$ is of degree $n$ if and only if $f(\sigma\cdot \infty)\neq 0$.
    \item If $f$ is irreducible, then $P_\sigma(f)$ is irreducible in $\mathbb{F}_q[x]$
    \item If $f$ satisfies $f(\sigma \cdot\infty)\neq 0$  and $\alpha_1,\ldots , \alpha_n$ are the roots of $P_\sigma(f)$, then the roots of $f$ are given by $\beta_1,\ldots , \beta_n$ with 
    \begin{equation*}
        \beta_i =\sigma \cdot \alpha_i, \quad\text{   for }\quad  i=1,\ldots, n 
    \end{equation*}
\end{enumerate}
\end{proposition}
Using Theorem \ref{theorepre} and Proposition \ref{theopsigma}, we can find conditions such that  the polynomial $P_\sigma(f)$ will be normal over $\mathbb{F}_q[x]$. The following Lemma is motivated by [\cite{Gaothesis} Theorem 3.4.1].
\begin{lemma}[Main Lemma] \label{mainlemma}
Let $\sigma=\begin{pmatrix}
a & b\\ 
c & d
\end{pmatrix} \in {\rm GL}_2(\mathbb F_q)$ and  $f(x)=x^n-A\neq x-a/c$ be an irreducible polynomial in  $\mathbb{F}_q[x]$. Let $g=P_\sigma(f)$ and $\alpha$ a root of $g$. Then:
\begin{enumerate}[i)]
    \item If $a=0$, then $g$ is normal if and only if $d\neq 0$, $n=p$, with $p$ prime and is $q$ a primitive root  $\mod p$.
    \item If $c=0$, then $g$ is normal if and only if $b\neq 0$, $n=p$, with $p$ prime and $q$ is  a primitive root  $\mod
    p$.
    \item If $ac\neq 0$, then $g$ is normal if and only if $c^{n-1}dA-ba^{n-1}\neq 0$ if and only if the trace of $\alpha$ over $\mathbb{F}_q$ is not zero.
\end{enumerate}
\begin{proof}
Since $f$ is irreducible and different from  $x-a/c$, then $f(\sigma\cdot \infty)\neq 0$. Let $\alpha$ be a root of $g$, by Proposition \ref{theopsigma} $iii)$, then $\theta=\sigma\cdot\alpha$ is a root of $f(x)=x^n-A$ and $\sigma^{-1}\cdot\theta=\alpha$.\\
Suppose first  that $ac=0$. If $a=0$, then $c\neq 0$ since $\det\sigma\neq 0$, so ${\sigma^{-1}=\frac{1}{-bc}\begin{pmatrix}
d & -b\\ 
-c & 0
\end{pmatrix}}$. Therefore
\begin{equation*}
\alpha=\frac{d\theta-b}{-c\theta}=\frac{dA-b\theta^{n-1}}{-cA} = -\frac{d}{c}+\frac{b}{cA}\theta^{n-1}= \alpha_0+\alpha_{n-1} \qquad \alpha_0 \in V_0, \alpha_1\in V_{n-1},
\end{equation*}
where $V_0$ and $V_{n-1}$ are the sets defined in Theorem \ref{theorepre}. It follows that $\alpha$ is normal if and only if $d\neq 0$  and $M_{n-1}=\{1,2,\ldots,n-1\}$. Since
\begin{equation*}
    |M_{n-1}|=\min \{m \in \mathbb{N}|(n-1)q^{m-1}\equiv n-1 \mod n \}= \min \{m \in \mathbb{N}|q^{m-1}\equiv 1 \mod n\}
\end{equation*}
 and $(q,n)=1$,  then $|M_{n-1}|$ divides $\phi(n)$, where $\phi$ is the Euler's totient function. In particular ${|M_{n-1}|\leq \phi(n)}$  so $M_{n-1}=\{1,2,\ldots,n-1\}$ iff $|M_{n-1}|=n-1\leq \phi (n)$ iff $n=p$ prime and $q$ is a primitive root  $\mod n$. This proves $i)$.\\
If $c=0$, then $a\neq 0$ since $\det \sigma \neq 0$, so ${\sigma^{-1}=\frac{1}{ad}\begin{pmatrix}
d & -b\\ 
0 & a
\end{pmatrix}}$ and 
\begin{equation*}
    \alpha=\frac{d\theta -b}{a}=-\frac{b}{a}+\frac{d}{a}\theta=\alpha_0+\alpha_1 \qquad \alpha_0\in V_0, \alpha_1 \in V_1
\end{equation*}
Then using Theorem \ref{theorepre} we see that $\alpha$ is normal if and only if $b\neq 0$ and $M_1=\{1,2\ldots, n-1\}$ and by a similar argument this is equivalent to $n=p$ with $p$ prime and $q$  a primitive root  $\mod n$, so this proves  $ii)$.

Now suppose that $ac\neq 0$. Since $f(\sigma\cdot \infty)\neq 0$ i.e. $a^n-c^nA\neq 0$, then
\begin{align*}
    \alpha&= \frac{d\theta-b}{-c\theta+a} \cdot \left( \frac{\sum_{i=0}^{n-1}a^{i}(c\theta)^{n-1-i}}{\sum_{i=0}^{n-1}a^{i}(c\theta)^{n-1-i}}\right)\\ &=\frac{\sum_{i=1}^{n-1}\theta^i[da^{n-i}c^{i-1}-bc^{i}a^{n-1-i}]+c^{n-1}dA-ba^{n-1}}{a^n-c^nA} \\&=\frac{\sum_{i=1}^{n-1}\theta^i(c/a)^{i-1}a^{n-2}\det\sigma +c^{n-1}dA-ba^{n-1}}{a^n-c^nA} \\ &=\sum_{l\in S}\alpha_l,
\end{align*}
where
\begin{equation*}
\alpha_l=\begin{cases}
\dfrac{a^{n-2}\det \sigma}{a^n-c^nA}\displaystyle  \sum_{i\in M_l}\theta^i(c/a)^{i-1} & \text{ if}\quad l\neq 0\vspace{0.22cm}\\ 
 \dfrac{c^{n-1}dA-ba^{n-1}}{a^n-c^nA}& \text{ if }\quad l=0
\end{cases}.
\end{equation*}
\textbf{Claim}: $\alpha_l\neq 0$ for all $l\neq 0$.

Assume for contradiction that $\alpha_l=0$ for some $l$ different from zero. Then 
\begin{equation*}
    \frac{a^{n-2}\det \sigma}{a^n-c^nA} \sum_{i\in M_l}\theta^i(c/a)^{i-1}=0, 
\end{equation*}
    and since $a,c,\det \sigma \neq 0$, this relation defines  a polynomial in $\mathbb{F}_q[x]$ with degree less than $n$ having  $\theta$ as a root, but $f(x)= x^n-A$ is irreducible and $f(\theta)=0$, so  $\alpha_l\neq 0$.
    
   By using Theorem \ref{theorepre}, we conclude that $\alpha$ is normal if and only if 
   \begin{equation*}
        \alpha_0=\frac{c^{n-1}dA-ba^{n-1}}{a^n-c^nA}\neq 0,
   \end{equation*}
   if and only if $c^{n-1}dA-ba^{n-1}\neq 0$.

   Otherwise, since $g$ is irreducible by Theorem \ref{theopsigma}, $g$ is the minimal polynomial for $\alpha$ over $\mathbb{F}_q[x]$ and 
   \begin{equation*}
   \begin{split}
       g&=P_\sigma(x^n-A)\\&=(ax+b)^n-A(cx+d)^n\\ &=x^n(a^n-c^nA)+x^{n-1}(nba^{n-1}-nc^{n-1}dA)+\text{ terms of degree }  <(n-1),
       \end{split}
   \end{equation*}
  then the trace of $\alpha$ over $\mathbb{F}_q$ is $n\left( \frac{c^{n-1}dA-ba^{n-1}}{a^n-c^nA}\right)$ and since $x^n-A$ is irreducible, necessary $(n,q)=1$. This relation  proves $iii)$.
\end{proof} 
\end{lemma}
To use  our Main Lemma in the proof of Theorem \ref{theonormal}, we need the following Lemma.
\begin{lemma} \label{fk}
Let  $t\geq 2,\sigma\in \rm{GL}_2(\mathbb{F}_q)$ and $g_0\neq x-a/c$ be as in Theorem \ref{Rsigma}. For $k\geq 0$ define $f_k=P_\sigma(g_k)$. Then, $f_k$ is  irreducible. Moreover, we have that  $f_k(x)=f_0(x^{t^k})$ and $P_{\sigma^{-1}}(f_k)=g_k$.  
\begin{proof}
See proof of Theorem \ref{Rsigma} in \cite{bassa2019rtransform}.
\end{proof}
\end{lemma}

 Now, we are ready to prove  Theorem  \ref{theonormal}.
\begin{proof}[Proof of Theorem \ref{theonormal}]
Let $\alpha_k\in \overline{\mathbb{F}}_q$ be any root of $g_k$ and let $s|t^k$ be a divisor. The first step of the proof is to show that $\alpha_k$ is root of a polynomial of the form $P_{\sigma^{-1}}(x^{t^k/s}-A_{k,s})$, with $x^{t^k/s}-A_{k,s}\in \mathbb{F}_{q^s}[x]$  irreducible.\\
Since $g_0=x-A\neq x -a/c$, then
\begin{equation*}
    f_0=P_\sigma(g_0)=(ax+b)-(cx+d)(A)=(a-Ac)(x-\sigma^{-1}\cdot A)=\lambda(x-\sigma^{-1}\cdot A),
\end{equation*}
where $\lambda\in \mathbb{F}_{q}^*$ ($g_0\neq x-a/c$). Therefore by Lemma \ref{fk}, $f_k(x)=\lambda(x^{t^k}-\sigma^{-1}\cdot A)$ is irreducible  and
\begin{equation*}
    g_k=P_{\sigma^{-1}}(f_k)=P_{\sigma^{-1}}(\lambda(x^{t^k}-\sigma^{-1}\cdot A))=\lambda P_{\sigma^{-1}}(x^{t^k}-\sigma^{-1}\cdot A).
\end{equation*}
Also, we note that $g_k(\sigma\cdot \infty)\neq 0$ for all $k\geq 0$. Indeed, if $k=0$, then $g_0(\sigma\cdot \infty)\neq 0$ because $g_0\neq x-a/c$. For $k\geq 1$, $g_k$ is an irreducible polynomial of degree bigger than 1, so it cannot have a root in $\mathbb{F}_q$. In the same way, we have that $f_k(\sigma^{-1}\cdot\infty)\neq 0 $ for all $k\geq 0$. Therefore, applying Proposition \ref{theopsigma} $iii)$ and Lemma \ref{fk}, we have that $\theta_k:=\sigma^{-1}\cdot \alpha_k$ is a root of $f_k$.\\
Moreover, the  polynomial    $x^{t^k/s}-(\theta_k)^{t^{k}/s}$ is irreducible in $\mathbb{F}_{q^s}[x]$ and has  $\theta_k$ as a root. Indeed, the polynomial $x^s-A\in \mathbb{F}[x]$ has $\theta^{t^k/s}$ as a root, but $f_k(x)=\lambda(x^{t^k}-\sigma^{-1}\cdot A)$ irreducible, then by  counting degrees we have that the polynomials  $x^{t^k/s}-(\theta_k)^{t^{k}/s}$ and $x^s-A\in \mathbb{F}[x]$ are  irreducible in $\mathbb{F}_{q^s}[x]$  and $\mathbb{F}_q[x]$ 
respectively. Therefore, using again Proposition \ref{theopsigma}  $iii)$, we have that $\alpha_k$ is root of  the polynomial $P_{\sigma^{-1}}(x^{t^k/s}-(\theta_k)^{t^k/s})$. Taking $A_{k,s}=(\theta_k)^{t^k/s}$ we prove the first step. 

The second step is to show that, under the assumptions of Theorem \ref{theonormal}, the polynomial $P_{\sigma^{-1}}(x^{t^k/s}-(\theta_k)^{t^k/s})$ is normal over $\mathbb{F}_{q^s}[x]$, so we will have that $\alpha_k$ is normal in the extension  $\left.\mathbb{F}_{q^{t^k}}\right/\mathbb{F}_{q^s}$.\\ Assume for contradiction that $P_{\sigma^{-1}}(x^{t^k/s}-(\theta_k)^{t^k/s})$ is not normal over $\mathbb{F}_{q^s}[x]$. Since $dc\neq 0$, applying  Lemma \ref{mainlemma} $iii)$ with $\sigma^{-1}$ and $x^{t^k/s}-(\theta_k)^{t^k/s}$, we have that
\begin{equation}\label{p}
    (-c)^{t^k/s-1}a(\theta_k)^{t^k/s}+bd^{t^k/s-1}=0.
\end{equation}
To prove $i)$, assume first that $b=0$. Then equation \ref{p} becomes
\begin{equation*}
    (-c)^{t^k/s-1}a(\theta_k)^{t^k/s}+bd^{t^k/s-1}= (-c)^{t^k/s-1}a(\theta_k)^{t^k/s}=0,
\end{equation*}
but this is true iff $c=0$, $a=0$ or $\theta_k=0$ (equivalently $\sigma^{-1}\cdot A=0$). Now $c=0$ cannot be true because $dc\neq 0$. The relation $a=0$ cannot be true either since $\det \sigma \neq 0$. If $\sigma^{-1}\cdot A=0$, then the polynomials $f_k(x)=x^{t^k}-\sigma^{-1}(A)=x^{t^k}$ are not irreducible, but this is a contradiction by Lemma \ref{fk}. Then the polynomial  $P_{\sigma^{-1}}(x^{t^k/s}-(\theta_k)^{t^k/s})$ is normal  over $\mathbb{F}_{q^s}[x]$. 

Since $s$ is an arbitrary divisor of $t^k$, we have that  $\alpha_k$ is completely normal in the extension $\left.\mathbb{F}_{q^{t^k}}\right/\mathbb{F}_{q}$ for all $k\geq 0$.\\
If $a=0$, then equation \ref{p}  becomes   $bd^{t^k/s-1}=0$ and by  a similar argument this is a contradiction. So the conclusion is the same as $b=0$. This proves $i)$.\\
Now, assume that $ab\neq 0$. We rewrite  equation \ref{p} as
\begin{equation}\label{p1}
      (\theta_k)^{t^k/s}=-\left(-\frac{d}{c}\right)^{t^k/s-1}\frac{b}{a}=\left(-\frac{d}{c}\right)^{t^k/s}\frac{bc}{ad}.
\end{equation}
\textbf{Claim}: $\alpha_k$ is normal in the extension $\left.\mathbb{F}_{q^{t^k}}\right/\mathbb{F}_{q^s}$ for each $s|t^k$ divisor different from 1. 

This  comes from of the fact that $x^{t^k}-\sigma^{-1}\cdot A$ is the minimal polynomial of $\theta_k$ over $\mathbb{F}_q[x]$, therefore necessarily $(\theta_k)^{t^k/s}\not\in\mathbb{F}_q$ for all $s|t^k$ divisor different of $1$, but this is a contradiction with the  equation \ref{p1} because $\sigma\in \rm{GL}_2(\mathbb{F}_q)$.\\
 For $s=1$,  equation \ref{p1} is equivalent to $\sigma^{-1}\cdot A=\left(-\frac{d}{c}\right)^{t^k}\frac{bc}{ad}$, and using 
\begin{equation*}
    \eta(g_0,\sigma)=\left(-\frac{d}{c}\right)\cdot g_0\left(\frac{b}{d}\right)\cdot g_0\left(\frac{a}{c}\right)^{-1}=\left(-\frac{d}{c}\right)\left(\frac{b}{d}-A\right)\left(\frac{a}{c}-A\right)^{-1}=\sigma^{-1}\cdot A,
\end{equation*}
 we have that  $\eta(g_0,\sigma)=\left(-\frac{d}{c}\right)^{t^k}\frac{bc}{ad}$.  \\
Assume that  $\frac{bc}{ad}$ is a $\ell$-th power for some $\ell$ dividing $t$. Then, $\eta(g,\sigma)$  is a $\ell$-th power, but this is a contradiction since $g$ satisfies the conditions of Theorem \ref{Rsigma}. Therefore $g_k$ is normal over $\mathbb{F}_q[x]$ and by the previous  Claim, this polynomial is completely normal over  $\mathbb{F}_q[x]$. This proves $ii)$.\\
Now, assume that $d=-c$, and $A\neq 0$. Then, $\sigma\cdot (\sigma^{-1}\cdot A)=A\neq 0$, so  $\sigma^{-1}\cdot A \neq  -b/a$ if and only if $a(\sigma^{-1}\cdot A)+b\neq 0$. Therefore
\begin{equation*}
    (-c)^{t^k-1}a(\sigma^{-1}\cdot A)+bd^{t^k-1}=d^{t^k-1}[a(\sigma^{-1}\cdot A)+b]\neq 0
\end{equation*}
for all $k\geq 0$. Then $g_k$ is normal over $\mathbb{F}_q[x]$ for all $k$ and by the previous Claim, this polynomial is completely normal over $\mathbb{F}_q[x]$. Thus proving $iii)$.
\end{proof}

\section{Chapman's Theorem}
In this section, we give a proof of Chapman's Theorem using Theorem \ref{theonormal}.\\ We recall that the $R$-transform corresponds to the case $t=2$ and $\sigma^*=\begin{pmatrix}
1 & 1\\ 
1 & -1
\end{pmatrix}$ in the $R_{\sigma,t}$-transform.

\begin{proof}[Proof of Chapman's Theorem \ref{chaptheo}]
We will show that there exist a linear polynomial $g_0\in \mathbb{F}_q[x]$, such that $g_1=g_0^R$ and it satisfies the conditions of Theorem \ref{Rsigma}. \\
Since $g_1$ is an irreducible self reciprocal polynomial, its roots are given by $\gamma$ and $1/\gamma$, so the element $A=(\gamma+\gamma^{-1})/2=-a/2$ belongs to $\mathbb{F}_q$.\\
\textbf{Claim}: The polynomial $g_0=x-A\in \mathbb{F}_q[x]$ satisfies the required conditions.\\
If $g_0=x-1$, then $a=-2$ but this is a contradiction because  $g_1=x^2+ax+1$ is irreducible.  The polynomial $g_0^R$ is monic and degree 2. Moreover, using that $g_0(A)=0$ and 
\begin{equation*}
    g_0^R=(2x)g_0\left(\dfrac{x+1/x}{2}\right),
\end{equation*}
we have that $g_0^R$ has $\gamma$ and $1/\gamma$ as its roots, this proves $g_0^R=g_1$. Also
\begin{equation*}
    \eta(g_0;\sigma)=\sigma^{-1}\cdot A=\frac{-1-A}{-1+A}=\frac{-2+a}{-2-a}=\left(\frac{1}{2+a}\right)^2\left(-4+a^2\right), 
\end{equation*}
and since $-4+a^2$ is minus the discriminant of the irreducible polynomial $g_1$ and $-1$ is a square in $\mathbb{F}_q$ ($q \equiv 1 \mod 4$), then the element $\eta(g_0;\sigma)$ is a quadratic nonresidue in $\mathbb{F}_q$, so $g_0$ satisfies the conditions of  Theorem  \ref{Rsigma}.
Then we conclude using  Theorem \ref{theonormal} $ii)$
\end{proof}

\section{Generalization of Theorem \ref{theonormal}}
In this section, we show a generalization of Theorem \ref{theonormal}. This result allows us to construct  families of completely normal polynomials $\{g_0\}_{k\geq 0}$ with the starting polynomial $g_0$  not necessarily  linear. We need first the followings results.
\begin{proposition}
 [\cite{cohen1969}, Theorem 1]\label{theoSt}
Let $A\in \mathbb{F}_q^*$ with multiplicative order $e$. Then the polynomial $x^t-A$ is irreducible if and only if the integer $t$ satisfies the following conditions:  
\begin{enumerate}[i)]
    \item $(t,(q-1)/e)=1$
    \item For all prime  factors $\ell$ of $t$, we have that $\ell$ divides $e$
    \item If $4|n$, then $4|q-1$
\end{enumerate}
\end{proposition}
Unlike  the proof of  Theorem \ref{theonormal}, we will use the following Theorem about the irreducibility  of the polynomials  $\{g_k\}_{k\geq 0}$.
\begin{theorem}[\cite{bassa2019rtransform},Theorem 3.3] \label{Rtgen}
Let $g_0\in \mathbb{F}_q[x]$ be an irreducible polynomial of degree $n$ and let $t\geq 2$ be a positive integer. If $q\equiv 3 \mod 4$ and $t$ is even assume moreover that $n$ is even. For $k\geq 1$ define $g_{k}=g_{k-1}^{R_{\sigma,t}}$. Assume that 
\begin{enumerate}[i)]
    \item $g_0(\sigma\cdot \infty)\neq 0$.
    \item $g_1$ is irreducible.
\end{enumerate}
Then, $\{g_k\}_{k\geq 0}$ is a sequence of irreducible polynomials  over $\mathbb{F}_q[x]$ with $\deg g_k=t^kn$.

\end{theorem} 
\begin{theorem} \label{teogen}
Suppose that the polynomial  $f_0(x)=x^n-A$ is irreducible over $\mathbb{F}_q[x]$. If $t$  is even, assume moreover that $q \equiv 1 \mod 4$. If the polynomial $g_0=  P_{\sigma^{-1}}(f_0)$ satisfies the conditions of Theorem \ref{Rtgen} with $dc\neq 0$, then we have the following results.
\begin{enumerate}[i)]
    \item If $ab=0$, then  $\{g_k\}_{k\geq 0}$ is a sequence of completely normal polynomials over $\mathbb{F}_q[x]$.
    \item If $ab\neq 0$, and the element $\frac{bc}{ad}$ is an $\ell$-th  power for some prime $\ell$ dividing $t n$, then  $\{g_k\}_{k\geq 0}$ is a sequence of completely normal polynomials over $\mathbb{F}_q[x]$.
    \item If $ab\neq 0$, $d=-c$ and $A\neq -b/a$, then $\{g_k\}_{k\geq 0}$ is a sequence of completely normal polynomials over $\mathbb{F}_q[x]$.
\end{enumerate}
\begin{proof}
Let $\alpha_k$ be a root of $g_k$. Similarly as in proof of the Theorem \ref{theonormal}, for each $s$ divisor of $t^k$, $\alpha_k$ is root of the polynomial $P_{\sigma^{-1}}(x^{nt^k/s}-(\theta_k)^{nt^k/s})$, with $\theta_k$ root of $f_k(x)=x^{t^k}-A$ and  $x^{nt^k/s}-(\theta_k)^{nt^k/s}\in \mathbb{F}_{q^s}[x]$ irreducible.\\
Since $dc\neq 0$, applying Lemma \ref{mainlemma} $iii)$, we have that $\alpha_k$ is normal in the extension $\left.\mathbb{F}_{q^{nt^k}}\right/\mathbb{F}_{q^s}$ if and only if
\begin{equation*}
    (-c)^{nt^k/s-1}aA+bd^{nt^k/s-1}\neq 0.
\end{equation*}
The proof of $i)$ and $iii)$ is the same as in Theorem \ref{theonormal}. For $ii)$ we only have to show that $g_k$ is normal over $\mathbb{F}_q[x]$. Since $ab\neq 0$, the condition
\begin{equation*}
    (-c)^{nt^k-1}aA+bd^{nt^k-1}\neq0,
\end{equation*}
is equivalent to
\begin{equation*}
    A\neq \left(-\frac{d}{c}\right)^{nt^k}\frac{bc}{ad}.
\end{equation*}
Let $e$ be the multiplicative order of $A$, since $f_0(x)=x^n-A$ and $f_1(x)=x^{nt}-A$ are irreducible, by Proposition \ref{theoSt} $i)$, we have that $(n,(q-1)/e)=1$ and $(nt,(q-1)/e)=1$, therefore $A$ is not an $\ell$-th power in $\mathbb{F}_q$ for all prime $\ell|nt$. So, if the element  $\frac{bc}{ad}$ is an $\ell$-th power in $\mathbb{F}_q$ with $\ell|nt$, then necessarily  $A\neq (-\frac{d}{c})^{nt^k}\frac{bc}{ad}$, it follows that $g_k $ is normal over $\mathbb{F}_q[x]$. This proves $ii)$.
\end{proof}
\end{theorem}

We can construct explicitly families of completely normal polynomials using Theorem \ref{teogen} and the following Lemma.
\begin{lemma}
Suppose that the polynomial $x^{nt}-A$ with $A\neq -d/c$ is irreducible over $\mathbb{F}_q[x]$. If $t$  is even assume moreover that $q \equiv 1 \mod 4$. Then the polynomial $g_0=  P_{\sigma^{-1}}(x^n-A)$ satisfies the conditions of Theorem \ref{Rtgen}.
\end{lemma}
\begin{proof}
We define  $f_0(x)=x^n-A$. Since $f_0\neq x-a/c$ is irreducible, then $f_0(\sigma^{-1}\cdot \infty)\neq 0$, so  applying Proposition \ref{fk} $i)$ and $ii)$ with $g_0=P_{\sigma^{-1}}(f_0)$, we have that $g_0$ is irreducible and $g_0(\sigma\cdot \infty)\neq 0$. Similarly we have that  $P_{\sigma^{-1}}(f_1)$  irreducible and by a direct calculation, this polynomial correspond to $g_1$. This proves the Lemma.
\end{proof}
For example, using Proposition \ref{theoSt}, the polynomial $x^{6}-3$ is  irreducible in $\mathbb{F}_7[x]$. Then for $t=3$ and $\sigma=\begin{pmatrix}
1 & 2 \\ 
2 & 1
\end{pmatrix}\in \rm{GL}_2(\mathbb{F}_q)$, we have that the element $\frac{bc}{ad}=4$ is an square in $\mathbb{F}_7$ and $\sigma
^{-1}=\begin{pmatrix}
2 & 3 \\ 
3 & 3
\end{pmatrix}$. So, applying Theorem \ref{teogen} $ii)$, with $n=2$ and  $g_0(x)=P_{\sigma^{-1}}(x^2-3)$, we have a family $\{g_k\}_{k\geq 0}$ of completely normal polynomials over $\mathbb{F}_7[x]$ with $\deg g_k=2\cdot 3^k$ and the starting polynomial
\begin{equation*}
    g_0(x)=P_\sigma^{-1}(x^2-3)=(2x+3)^2-3(3x+2)^2=5x^2+4x+4.
\end{equation*}
For $\mathbb{F}_{31}$, $\rm{ord}(5)=3$ and $\rm{ord}(2)=5$. So, the polynomial $x^{15}-10$ is  irreducible in $\mathbb{F}_{31}[x]$. Taking $t=5$ and $\sigma=\begin{pmatrix}
2 & 2 \\ 
2 & -2
\end{pmatrix} \in \rm{GL}_2(\mathbb{F}_{31})$, we have that $A=10\neq -2/2=-1 $ and $\sigma^{-1}=\begin{pmatrix}
8 & 8 \\ 
8 & 23
\end{pmatrix}$. Then, using Theorem \ref{teogen} $iii)$ with $g_0=P_{\sigma^{-1}}(x^3-10)$, we get a family $\{g_k\}_{k\geq0}$ of completely normal polynomials with $\deg g_k=3\cdot 5^k$ and starting polynomial
\begin{equation*}
    g_0(x)=P_{\sigma^{-1}}(x^3-10)=(8x+8)^3-10(8x+23)^3=11x^3+x^2+2x+21
\end{equation*}
\section*{Acknowledgements}
 The author thanks  Ricardo Menares to motivate him to write this paper  and his invaluable comments on making this work more readable.
\bibliographystyle{unsrt}
\bibliography{references}
\end{document}